\documentclass{amsart}
\usepackage[all]{xy}
\usepackage{amssymb}
\usepackage{url}

\newcommand{\R}{\mathbb{R}}
\newcommand{\fC}{\mathfrak{C}}
\newcommand{\fB}{\mathfrak{B}}
\newcommand{\fG}{\mathfrak{G}}
\newcommand{\sG}{\mathsf{G}}
\newcommand{\boundary}{\partial}
\newcommand{\modsim}{/ \!\! \sim \,}
\newcommand{\canomega}{\omega_{\mathrm{can}}}
\newcommand{\canlam}{\lambda_{\mathrm{can}}}
\newcommand{\cycl}{\{\mathrm{cycl.}\}}
\newcommand{\LWX}{\mathrm{LWX}}

\newcommand{\cald}{{\mathcal {D}}}
\newcommand{\frakg}{\mathfrak{G}}
\newcommand{\vect}{\mathfrak{X}}
\DeclareMathOperator{\im}{im}
\newcommand{\tolabel}[1]{\xrightarrow{#1}}

\newtheorem{thm}{Theorem}[section]
\newtheorem{prop}[thm]{Proposition}
\newtheorem{lemma}[thm]{Lemma}
\newtheorem{cor}[thm]{Corollary}
\newtheorem{conj}[thm]{Conjecture}

\theoremstyle{definition}
\newtheorem{definition}[thm]{Definition}

\theoremstyle{remark}
\newtheorem{remark}[thm]{Remark}

\newtheorem*{ack}{Acknowledgements}
\newtheorem*{org}{Organization of the paper}
\numberwithin{equation}{section}

\begin{document}

\title[Integration of exact Courant algebroids]{Symplectic structures on the integration of exact Courant algebroids}
\author{Rajan Amit Mehta}
\address{Department of Mathematics \& Statistics\\
Smith College\\
44 College Lane\\
Northampton, MA 01063}
\email{rmehta@smith.edu}

\author{Xiang Tang}
\address{Department of Mathematics\\
Washington University in Saint Louis\\
One Brookings Drive\\
Saint Louis, Missouri, USA 63130}
\email{xtang@math.wustl.edu}

\begin{abstract}
We construct an infinite-dimensional symplectic $2$-groupoid as the integration of an exact Courant algebroid. We show that every integrable Dirac structure integrates to a ``Lagrangian" sub-$2$-groupoid of this symplectic $2$-groupoid. As a corollary, we recover a result of Bursztyn-Crainic-Weinstein-Zhu that every integrable Dirac structure integrates to a presymplectic groupoid.
\end{abstract}

\dedicatory{Dedicated to Professor Alan Weinstein on the occasion of his $70$th birthday}


\maketitle

\section{Introduction}
In the late 80's, T. Courant  and A. Weinstein \cite{cw:courant} introduced the notion of \emph{Dirac structure} as a way of unifying Poisson, symplectic, and presymplectic structures. An important ingredient in the definition of Dirac structure is a bracket, now called the \emph{Courant bracket}, defined on the direct sum of the spaces of vector fields and $1$-forms on a manifold. In the 90's, Z. Liu, Weinstein, and P. Xu \cite{lwx:courant} formalized the properties of the Courant bracket in the definition of a \emph{Courant algebroid}.

For any manifold $M$, the standard Courant algebroid over $M$ is the bundle $TM\oplus T^*M$, equipped with following structures on its space of sections $\vect(M) \oplus \Omega^1(M)$:
\begin{enumerate}
        \item the symmetric bilinear form given by
        \[       \langle X_1+\xi_1, X_2+\xi_2\rangle=\xi_2(X_1)+\xi_1(X_2),      \]
        and

        \item the Courant bracket, given by
        \[ [X_1+\xi_1, X_2+\xi_2]=[X_1, X_2]+ L_{X_1}\xi_2 - \iota_{X_2} d\xi_1
        \]
        for $X_i \in \vect(M)$, $\xi_i \in \Omega^1(M)$.
\end{enumerate}
The Courant bracket satisfies the Jacobi identity but is not skew-symmetric.

Given a closed $3$-form $H$ on $M$, one can define a ``twisted'' version of the Courant bracket by
\[
[X_1+\xi_1, X_2+\xi_2]_H =[X_1, X_2] + L_{X_1}\xi_2 - \iota_{X_2} d\xi_1 + \iota_X\iota_Y H.
\]
P. \v{S}evera \cite{severa:letters} proved that the cohomology class $[H]\in H^3(M)$ classifies \emph{exact Courant algebroids}, i.e. those that fit into an exact sequence 
\[
0\longrightarrow T^*M\longrightarrow E \longrightarrow TM\longrightarrow 0,
\]
up to isomorphism.

A (twisted) Dirac structure is a subbundle of $TM\oplus T^*M$ that is maximally isotropic with respect to the bilinear form $\langle \cdot, \cdot \rangle$ and whose sections are closed under the (twisted) Courant bracket. When the Courant bracket is restricted to a Dirac structure, the skew-symmetry anomaly disappears, making the Dirac structure into a Lie algebroid. In \cite{bcwz:dirac}, H. Bursztyn, M. Crainic, Weinstein, and C. Zhu showed that, if the Lie algebroid associated to a Dirac structure is integrable in the sense of \cite{cf:integration}, then the Lie groupoid carries a natural closed (or $H$-closed, in the twisted case), multiplicative $2$-form, making it a \emph{presymplectic groupoid}.

In this article, we study the integration problem for Courant algebroids, which was one of open problems raised by Liu, Weinstein, and Xu \cite{lwx:courant}:
\vskip 2mm

\noindent{``{\bf Open Problem 5.} \em What is the global, groupoid-like object corresponding to a Courant algebroid? In particular, what is the double of a Poisson groupoid?''}
\vskip 2mm
From the work of \v{S}evera \cite{severa:some}, it is expected that the solution to this problem is a \emph{symplectic $2$-groupoid}. Recently, the authors \cite{mt:double}, D. Li-Bland and P. S\v{e}vera \cite{ls:integration}, and Y. Sheng and C. Zhu \cite{sz:integration} independently constructed $2$-groupoids integrating certain subclasses of Courant algebroids, with the standard Courant algebroid being the common element of all three subclasses.  For these constructions, the term ``integration'' can be justified by showing that the Courant algebroid structure can be recovered via \v{S}evera's $1$-jet construction \cite{severa:linfty}. 

An important question that has remained unaddressed is how a symplectic $2$-groupoid integrating a Courant algebroid is related to the presymplectic groupoids integrating the Dirac structures that sit inside the Courant algebroid. In fact, one can easily find examples showing that the symplectic $2$-groupoids of \cite{mt:double}, \cite{ls:integration}, and \cite{sz:integration} are not large enough to contain all the presymplectic groupoids arising from Dirac structures. In this article, we show that this problem can be resolved, at the cost of working with infinite-dimensional manifolds.

We construct, for any manifold $M$, an infinite-dimensional Lie $2$-groupoid, i.e. a Kan simplicial (Banach) manifold $\{X_\bullet\}$ for which the horn fillings are unique in degrees greater than $2$. For any closed $H \in \Omega^3(M)$, we obtain a natural multiplicative symplectic $2$-form $\omega^H_2$ on $X_2$, making $\{X_\bullet\}$ into a symplectic $2$-groupoid, which we call the \emph{Liu-Weinstein-Xu $2$-groupoid}, or $\LWX(M)$ for short. 

A brief description of $\LWX(M)$ in low degrees is as follows. The space of ``$0$-simplices'' is $\LWX_0(M) = M$. The space $\LWX_1(M)$ of ``$1$-simplices'' consists of bundle maps from the tangent bundle of the standard $1$-simplex to $T^*M$. An element of the space $\LWX_2(M)$ of ``$2$-simplices'' is given by a quadruplet $([f], \psi_0, \psi_1, \psi_2)$, where $[f]$ is a class of maps from the standard $2$-simplex to $M$, modulo boundary-fixing homotopies, and where each $\psi_i$ is an element of $\LWX_1(M)$ whose base map is the $i$th edge of $f$. In order to endow $\LWX(M)$ with a smooth structure, we require the maps to have certain fixed orders of differentiability; the details are in Section \ref{sec:2groupoid}.

Our most significant results arise from the observation that $\LWX(M)$ has a natural symplectic $2$-form $\omega^H_1$ on $\LWX_1(M)$ for which
\begin{align}\label{eqn:ddeltah}
d\omega^H_1&=\delta H,& \delta \omega^H_1&=\omega^H_2,
\end{align}
where $\delta: \Omega^\bullet(\LWX_k(M))\to \Omega^\bullet(\LWX_{k+1}(M))$ is the simplicial coboundary map. This structure seems to be specific to the case of exact Courant algebroids, so it does not appear in the general theory sketched out in \cite{severa:some}.

We associate to any integrable Dirac structure a sub-$2$-groupoid of $\LWX(M)$ whose $1$-truncation can be identified with the Lie groupoid integrating the Dirac structure. We prove that the pullback of $\omega^H_2$ vanishes on this sub-$2$-groupoid; as a result, we can deduce that the pullback of $\omega^H_1$ descends to the $1$-truncation, inducing an $H$-closed, multiplicative $2$-form on the Lie groupoid integrating the Dirac structure. This $2$-form precisely coincides with the one constructed by Bursztyn, Crainic, Weinstein, and Zhu \cite{bcwz:dirac}. We can thus view the Liu-Weinstein-Xu $2$-groupoid as being the geometric origin of presymplectic groupoids.

We prove that the sub-$2$-groupoid associated to a Dirac structure is in fact Lagrangian at the ``units'' of the $2$-groupoid. We conjecture that it is Lagrangian everywhere, and we prove the conjecture in a special case. We believe that the Lagrangian property is the origin of the nondegeneracy condition in \cite{bcwz:dirac} and therefore deserves further study.

Another issue that we do not address here is that of the relationship between $\LWX(M)$ and the finite-dimensional symplectic $2$-groupoids of \cite{ls:integration,mt:double,sz:integration}. Clearly, there should be a notion of equivalence between symplectic $2$-groupoids, but the precise nature of the equivalence remains an open question.

\begin{org}
In Section \ref{sec:2groupoid}, we construct an infinite-dimensional simplicial manifold $\{\fC_\bullet(M)\}$ associated to any manifold $M$. We show that $\{\fC_\bullet(M)\}$ can be truncated to an infinite-dimensional Lie $2$-groupoid, which we denote $\LWX(M)$. In Section \ref{sec:symplectic}, we construct canonical symplectic forms $\omega_i$ on $\LWX_i(M)$ for $i=1,2$, as well as twisted versions $\omega_i^H$ associated to any closed $3$-form $H$ on $M$.  In particular, we show that the relations \eqref{eqn:ddeltah} are satisfied. In Section \ref{sec:dirac}, we construct a simplicial manifold $\{\fG_\bullet(\cald)\}$ associated to any Dirac structure $\cald$, whose $1$-truncation is the Lie groupoid $\sG$ integrating the Dirac structure. There is a natural inclusion map $\fG_\bullet(\cald) \hookrightarrow \fC_\bullet(M)$, and we show that the pullback of $\omega_2$ vanishes, implying that $\omega_1$ induces a presymplectic structure on $\sG$. Finally, in Section \ref{sec:lagrangian}, we show that the image of $\fG_2(\cald)$ in $\LWX(M)$ is Lagrangian at the units and conjecture that it is Lagrangian everywhere.
\end{org}

\begin{ack}
We would like to thank Professor Alan Weinstein for his continuing advice and encouragement during and since our Ph.\ D. study. Tang's research is partially supported by NSF grant 0900985, and NSA grant H96230-13-1-02.
\end{ack}

\section{Construction of $\LWX(M)$}\label{sec:2groupoid}
In this section, we describe the construction of the Liu-Weinstein-Xu $2$-groupoid, by first constructing an infinite-dimensional simplicial manifold, and then truncating it to obtain a Lie $2$-groupoid.

The basic idea of the construction in this section goes back to D. Sullivan \cite{sullivan}. The application of Sullivan's idea to the integration problems in Poisson geometry was described by \v{S}evera \cite{severa:letters, severa:some}, using the language of $NQ$-manifolds. In this section, we describe a direct construction that does not require any knowledge of supergeometry.

\subsection{A simplicial manifold}

Let $M$ be a manifold. Recall that, if $X$ is a manifold and $E \tolabel{\pi} M$ is a vector bundle, then a map $\phi: X \to E$ is said to be of class $C^{p,q}$ if $\phi$ is $C^q$ and $\pi \circ \phi$ is $C^p$. Clearly, if this is the case, then it is necessary that $p \geq q$.

Fix $p \geq q$. For each integer $n\geq0$, let $\fC_n(M)$ denote the set of $C^{p,q}$ bundle maps from $T\Delta^n$ to $T^*M$, where $\Delta^n$ is the standard $n$-dimensional simplex in $\R^n$.

\begin{lemma}\label{lem:banach-mfld} The space $\fC_n(M)$ of $C^{p,q}$ bundle maps from $T\Delta^n$ to $T^*M$ is a Banach manifold.  
\end{lemma}
\begin{proof} 
Let $T^{*}_nM$ be the $n$-fold direct sum of $T^*M$. That is,
\[
T^*_n M := \underbrace{T^*M\oplus\cdots \oplus T^*M}_n. 
\] 
Then we may use the standard trivialization $T\Delta^n = \Delta^n \times \R^n$ to obtain a one-to-one correspondence between bundle maps $\varphi:T\Delta^n\to T^*M$ and maps $\tilde{\varphi}:\Delta\to T^*_nM$. Specifically, given such a $\varphi$, we obtain $\tilde{\varphi}$ by evaluating $\varphi$ on the standard basis vectors of $\R^n$. This correspondence preserves order of differentiability, in that $\varphi$ is $C^{p,q}$ if and only if $\tilde{\varphi}$ is $C^{p,q}$. 

The statement immediately follows from the fact that the space of $C^{p,q}$-maps from $\Delta^n$ to $T^*_nM$ is a Banach manifold. 
\end{proof}

There is a cosimplicial manifold structure on $\{T\Delta^\bullet\}$, obtained by applying the tangent functor to the standard cosimplicial manifold $\{\Delta^\bullet\}$. Thus there is an induced simplicial manifold structure on $\{\fC_\bullet(M)\}$. For each $n$, we will use $d_i : \fC_n(M) \to \fC_{n-1}(M)$, $0 \leq i \leq n$, to denote the face maps and $s_i : \fC_n(M) \to \fC_{n+1}(M)$, $0 \leq i \leq n$ to denote the degeneracy maps.

We note that $\fC_0(M) = M$, and that $\fC_1(M) = \{ C^{p,q} \mbox{ bundle maps } T[0,1] \to T^*M\}$ can be identified with the space of $C^{p,q}$ paths on $T^*M$. 

\subsection{$2$-groupoid truncation} 
Recall that a simplicial manifold satisfies the \emph{Kan condition} if all the horn maps are surjective submersions. 

\begin{prop}\label{prop:kan}The simplicial manifold $\{\fC_\bullet(M)\}$ satisfies the Kan condition. 
\end{prop}
\begin{proof}For each $n$, let $S_n(M)$ denote the set of $C^p$ maps from $\Delta^n$ to $M$. It is known (see, for example, \cite[Lemma 5.7]{H}) that $\{S_\bullet(M)\}$ is a Kan simplicial manifold. There is a natural projection map $\{\fC_\bullet(M)\} \to \{S_\bullet(M)\}$, so, by \cite[Lemma 2.8]{H}, it suffices to show that this map is a Kan fibration. In this setting, the Kan fibration condition is as follows.

For $0 \leq \ell \leq n$, let $\Lambda_{n,\ell}$, applied to any simplicial manifold, denote the space of $n$-dimensional horns where the $\ell$th face is omitted. There is a natural map $\fC_n(M) \to S_n(M) \times_{\Lambda_{n,\ell}(S(M))} \Lambda_{n,\ell}(\fC(M))$, taking $\varphi \in \fC_n(M)$ to $(\bar{\varphi}, (d_0 \varphi, \dots, \widehat{d_\ell \varphi}, d_n \varphi))$, where $\bar{\varphi}$ is the underlying base map of $\varphi$. The Kan fibration condition requires that this map be a surjective submersion for all $\ell$ and $n$, which we will prove using a method similar to that used in \cite[Lemma 5.7]{H}.

Let $f: \Delta^n \to M$ be a $C^p$ map, and let $\psi_i: T\Delta^{n-1} \to T^*M$, $i \neq \ell$ be a collection of $C^{p,q}$ maps forming a horn in $\Lambda_{n,\ell}(\fC(M))$ that is compatible with $f$. For each nonempty $I \subset \{0, \dots, n\} \smallsetminus \{\ell\}$, let $F_I \subset \Delta^n$ denote the $(n-|I|)$-dimensional subface whose vertices are $\{0,\dots,n\} \smallsetminus I$. As a result of the horn compatibility conditions, the maps $\psi_i$ induce well-defined $C^{p,q}$ maps $\psi_I: TF_I \to T^*M$. 

For each $I$, let $p_I: \Delta^n \to F_I$ be the affine projection map collapsing the vertices in $I$ onto $\ell$. Fix a Riemannian metric on $M$. Then, for each $t \in \Delta^n$, we may use parallel transport along the image of the line from $t$ to $p_I(t)$ to identify $T^*_{f(t)}M$ with $T^*_{f(p_I(t))}M$. 

We now define a map $\varphi: T\Delta^n \to T^*M$ with base map is $f$, given by 
\[ \varphi = \sum_{I \subset \{0, \dots, n\} \smallsetminus \{\ell\}} (-1)^{|I|+1} \psi_I \circ Tp_I.\]
It is clear by construction that $\varphi$ is $C^{p,q}$, and it follows from the identities satisfied by the projection and face maps that $d_i \varphi = \psi_i$ for each $i \neq \ell$. This proves surjectivity.

To show that the map is a submersion, we observe that, under a sufficiently small change in $f$, we can use parallel transport to accordingly change any horn filling $\varphi$, and under a change in $\{\psi_i\}$, we can apply the above construction to the difference to accordingly change $\varphi$. This process gives a local section through any $\varphi \in \fC_n(M)$.
\end{proof}

Recall that an \emph{$n$-groupoid} is a Kan simplicial set for which the horn-fillings are unique in dimensions greater than $n$. A \emph{Lie $n$-groupoid} is a Kan simplicial manifold satisfying the same condition. 

Duskin \cite{duskin} introduced a \emph{truncation} functor $\tau_{\leq n}$ which may be applied to a Kan simplicial set $\{X_\bullet\}$ to produce an $n$-groupoid. It is defined as follows:
\begin{itemize}
\item $(\tau_{\leq n} X)_m=X_m$ for $m<n$;

\item $(\tau_{\leq n} X)_n=X_n \modsim$, where $x \sim y$ if and only if there exists $z \in X_{n+1}$ such that $d_n z = x$, $d_{n+1} z = y$, and $d_i z \in \im(s_{n-1})$ for $0 \leq i < n$;

\item $(\tau_{\leq n} X)_m=X_m \modsim$ for $m>n$, where $x \sim y$ if and only if the $n$-skeletons of $x$ and $y$ are equivalent with respect to the equivalence relation on $X_n$.
\end{itemize}

The following result can be found in \cite[Lemma 3.6]{H} and \cite[Section 2]{Z}. 

\begin{lemma}\label{lem:truncation} If $\{X_\bullet\}$ is a Kan simplicial manifold, then $\tau_{\leq n}X$ is an $n$-groupoid. If, furthermore, $X_n \modsim$ is a manifold, then $\tau_{\leq n}X$ is a Lie $n$-groupoid. 
\end{lemma}

By applying Lemma \ref{lem:truncation} to $\{\fC_\bullet(M)\}$ for $n=2$, we obtain the \emph{Liu-Weinstein-Xu $2$-groupoid} $\LWX(M) := \tau_{\leq 2} \fC(M)$. The main result of this section is the following:
\begin{thm}\label{thm:truncation}
The quotient $\mathfrak{C}_2(M) \modsim$ is a Banach manifold, and therefore $\LWX(M)$ is a Lie $2$-groupoid. 
\end{thm}

Section \ref{sec:truncationproof} is devoted to proving Theorem \ref{thm:truncation}.

\subsection{Proof of Theorem \ref{thm:truncation}}\label{sec:truncationproof}

Without loss of generality, we will assume that $M$ is connected.

\begin{lemma} \label{lem:anchorsub}
The map $\fC_1(M) \to M \times M$ given by $\psi \mapsto (d_0 \psi, d_1 \psi)$ is a surjective submersion.
\end{lemma}
\begin{proof}
Surjectivity follows from the assumption that $M$ is connected. To prove that the map is a submersion, we will describe a way to construct local sections.

Recall that $\fC_1(M)$ can be identified with the space of $C^{p,q}$ paths on $T^*M$. Let $\psi$ be such a path. Choose a Riemannian metric on $M$, and, for $i=1,2$, let $U_i$ be a neighborhood of $d_i \psi$ for which the exponential map is a diffeomorphism. Using the exponential map and parallel transport along the base path $\bar{\psi}:[0,1] \to M$, we may then identify the neighborhood $U_1 \times U_2$ of $(d_0 \psi, d_1 \psi)$ with a neighborhood of $(0,0)$ in $T_{d_0 \psi} M \times T_{d_0 \psi} M$.

For any $(v_0,v_1) \in T_{d_0 \psi} M \times T_{d_0 \psi} M$, we may (again using parallel transport) view $v(t) := (1-t)v_0 + tv_1$ as a vector field along $\bar{\psi}$. By exponentiating $\bar{\psi}$ in the direction of $v(t)$ and parallel transporting the cotangent vectors of $\psi$, we obtain a path $\psi' \in \fC_1(M)$ for which $(d_0 \psi', d_1 \psi') = (v_0,v_1)$. This process provides a well-defined local section through $\psi$.
\end{proof}

Let $\fB(M)$ be defined as the space of ``triangles'' of paths in $\fC_1(M)$. More precisely, $\fB(M)$ is the space of triples $(\psi_0,\psi_1,\psi_2)$, $\psi_i \in \fC_1(M)$, such that 
\begin{align}\label{eqn:triangle}
d_0 \psi_2 &= d_1 \psi_0, & & d_0 \psi_1 = d_0 \psi_0, & d_1 \psi_1 = d_1 \psi_2.
\end{align}

\begin{lemma}\label{lem:fbbanach}
$\fB(M)$ is a Banach manifold.
\end{lemma}
\begin{proof}
Recall (see, for example, \cite[Lemma 4.4]{H}) that Banach manifolds are closed under fiber products where one of the maps is a surjective submersion. Since the face maps of a simplicial manifold are surjective submersions, the horn space $\Lambda_{2,1}(\fC(M))$, consisting of pairs $(\psi_0,\psi_2) \in \fC_1(M) \times \fC_1(M)$ satisfying the first equation in \eqref{eqn:triangle}, is a Banach manifold.

We may view $\fB(M)$ as the fiber product over $M \times M$ of $\fC_1(M)$ and $\Lambda_{2,1}(\fC(M))$, where the fiber product imposes the latter two equations in \eqref{eqn:triangle}. Since this fiber product involves the surjective submersion $\fC_1(M) \to M \times M$ from Lemma \ref{lem:anchorsub}, it follows that $\fB(M)$ is a Banach manifold.
\end{proof}

Let $\pi_\fB: \fB(M) \to C^p(\boundary \Delta^2; M)$ be the map taking $(\psi_0,\psi_1,\psi_2) \in \fB(M)$ to its base map $(\bar{\psi}_0,\bar{\psi}_1,\bar{\psi}_2)$.

\begin{lemma}\label{lemma:pifbsub}
The map $\pi_\fB$ is a surjective submersion.
\end{lemma}
\begin{proof}
Surjectivity is clear, since the edges of any map from $\boundary \Delta^2$ to $M$ can be lifted to zero maps $T\Delta^1 \to T^*M$.

Choose a Riemannian metric on $M$. For any map $f: \boundary \Delta^2 \to M$, we can use the exponential map to identify any sufficiently close maps with lifts $\tilde{f}: \boundary \Delta^2 \to TM$. Given such a lift, we can use parallel transport along the exponential paths to translate any element of $\fB(M)$ whose base map is $f$. This process gives a local section of $\pi_\fB$.
\end{proof}

There is a natural ``$1$-skeleton'' map $\nu: \fC_2(M) \to \fB(M)$, given by $\nu(\varphi) = (d_0 \varphi, d_1 \varphi, d_2 \varphi)$. The map $\nu$ is invariant under the equivalence relation that defines $\LWX_2(M) = \fC_2(M) \modsim$.

Let $\fB_0(M)$ be the connected component of $\fB(M)$ consisting of elements $\beta$ for which $\pi_\fB(\beta)$ is contractible. Clearly, the image of $\nu$ is contained in $\fB_0(M)$. Thus, we see that $\nu$ induces a map $\hat{\nu} : \LWX_2(M) \to \fB_0(M)$.

There is another map $\pi_{\fC}: \fC_2(M) \to S_2(M) := C^p(\Delta^2;M)$, taking $\varphi$ to its base map $\bar{\varphi}$. An equivalence between elements $\varphi, \varphi' \in \fC(M)$ induces a boundary-fixing homotopy between $\bar{\varphi}$ and $\bar{\varphi}'$, so $\pi_{\fC}$ descends to a map from $\LWX_2(M)$ to $S_2(M)\modsim$, where the equivalence relation is boundary-fixing homotopy. We observe that $S_2(M)\modsim$ is a covering of the component of contractible maps in $C^p(\boundary \Delta^2;M)$ and is therefore a Banach manifold\footnote{This statement is a higher-dimensional analogue of the fact that the fundamental groupoid of a manifold $M$ is a cover of $M \times M$, and the proof is similar. We leave the details to the reader.}. In particular, if $\pi_2(M)=0$, then $S_2(M)\modsim = C^p(\boundary \Delta^2; M)$.

\begin{lemma}\label{lem:covering}
The map $(\hat{\nu}, \pi_\fC)$ is a bijection from $\LWX_2(M)$ to the fiber product (over $C^p(\boundary \Delta^2; M)$) of $\fB_0(M)$ with $S_2(M)\modsim$. 
\end{lemma}

\begin{proof}
Throughout this proof, we will assume that a choice of Riemannian metric on $M$ has been fixed, and we will implicitly use parallel transport to identify cotangent spaces at different points along paths in $M$.

We will first show that $(\hat{\nu},\pi_\fC)$ is surjective.  Let $(\psi_0,\psi_1,\psi_2)$ be in $\fB_0$, and let $f$ be a compatible map in $C^p(\Delta^2;M)$. Using the standard trivializations of $\Delta^1$ and $\Delta^2$, we can identify each $\psi_i$ with a path in $T^*M$, and we can identify $\fC_2(M)$ with the space of $C^{p,q}$ maps from $\Delta^2$ to $T^*_2M := T^*M \oplus T^*M$.

For $i=0,1,2$, let $\beta_i$ be the path in $T^*_2 M$ with the same base path as $\psi_i$, given by
\begin{align*}
\beta_0(t) &= \left(\psi_0(t), (1-t)(\psi_1(0) - \psi_0(0)) + t\psi_2(0)\right),\\
\beta_1(t) &= \left((1-t)\psi_0(0) + t(\psi_1(t) - \psi_2(1)), (1-t)(\psi_1(t) - \psi_0(0)) + t\psi_2(1)\right),\\
\beta_2(t) &= \left((1-t)\psi_0(1) + t(\psi_1(1) - \psi_2(1)), \psi_2(t)\right).
\end{align*}
These paths agree at the endpoints, in that $\beta_2(0) = \beta_0(1)$, $\beta_0(0) = \beta_1(0)$, and $\beta_1(1) = \beta_2(1)$, so they form a well-defined map $\beta: \boundary \Delta^2 \to T^*_2 M$ that can be extended to a $C^q$ map $\varphi: \Delta^2 \to T^*_2 M$ for which the base map is $f$. By construction, we have that $\hat{\nu}(\varphi) = (\psi_0,\psi_1,\psi_2)$ and $\pi_\fC(\varphi) = f$. 

Next, we will show that $(\hat{\nu},\pi_\fC)$ is one-to-one. Suppose that $\varphi, \varphi'$ are elements of $\fC_2(M)$ for which $\nu(\varphi) = \nu(\varphi')$ and $\pi_\fC(\varphi) \sim \pi_\fC(\varphi')$. By a process similar to the proof of surjectivity, one can construct a map from $\boundary \Delta^3$ to $T^*_3 M$ which, if it could be extended to a map $\zeta: \Delta^3 \to T^*_3 M$, would satisfy the conditions of an equivalence between $\varphi$ and $\varphi'$. The assumption that $\pi_\fC(\varphi)$ and $\pi_\fC(\varphi')$ are in the same boundary-fixing homotopy class guarantees that such an extension $\zeta$ does exist (and can be chosen to have the same order of differentiability as the boundary), proving that $\varphi$ and $\varphi'$ represent the same element of $\LWX_2(M)$. 
\end{proof}

Theorem \ref{thm:truncation} follows  directly from Lemmas \ref{lem:fbbanach}--\ref{lem:covering}.

\begin{remark}
If $\pi_2(M) = 0$, then Lemma \ref{lem:covering} implies that $\LWX_2(M)$ is naturally diffeomorphic to $\fB_0(M)$.
If $\pi_1(M) = 0$, then $\fB_0(M) = \fB(M)$. Thus, if $M$ is $2$-connected, then $\hat{\nu}$ is a diffeomorphism from $\LWX_2(M)$ to $\fB(M)$.  This fact provides a simple description (in the $2$-connected case) of elements of $\LWX_2(M)$ as triangles of paths in $\fC_1(M)$. For general $M$, an appropriate modification is as follows: an element of $\LWX_2(M)$ corresponds to a quadruplet $([f], \psi_0, \psi_1, \psi_2)$, where
\begin{itemize}
\item $[f]$ is a class of $C^p$ maps from $\Delta^2$ to $M$, modulo boundary-fixing homotopy, and
\item each $\psi_i$ is a $C^q$ lift of the $i$th edge of $f$ to $T^*M$.
\end{itemize}
\end{remark}

\section{Symplectic structures}\label{sec:symplectic}

In this section, we describe how the canonical symplectic form on $T^*M$ induces symplectic structures on $\LWX_1(M)$ and $\LWX_2(M)$.

\subsection{Multiplicative forms and truncation}
Let $X_\bullet$ be a Kan simplicial manifold, and let $\alpha$ be a differential form on $X_n$ for some $n$. Recall that the simplicial coboundary of $\alpha$ is defined as
\[ \delta \alpha := \sum_{i=0}^{n+1} (-1)^i d_i^* \alpha.\]
We say that $\alpha$ is \emph{multiplicative} if $\delta \alpha = 0$.

\begin{prop}\label{prop:basic}
If $\alpha$ is multiplicative, then $\alpha$ is basic with respect to the quotient map $X_n \to (\tau_{\leq n} X)_n$. 
\end{prop}

\begin{proof}
Recall that the quotient is defined by the equivalence relation where $x \sim y$ if and only if there exists $z \in X_{n+1}$ such that $d_n z = x$, $d_{n+1} z = y$, and $d_i z \in \im(s_{n-1})$ for $0 \leq i < n$.  In this case, it follows that $d_i z = s_{n-1} d_i x = s_{n-1} d_i y$ for $0 \leq i < n$. 

From the definition of the equivalence relation, we can see that a vector $v \in TX_n$ is tangent to a fiber of the quotient map if and only if there exists a vector $\tilde{v} \in TX_{n+1}$ such that $(d_{n+1})_* \tilde{v} = v$ and $(d_i)_* \tilde{v} = 0$ for $0 \leq i \leq n$. If this is the case, then, for any $\alpha \in \Omega^k (X_n)$,
\[ \alpha(v, \cdot, \dots, \cdot) = \pm (\delta \alpha)(\tilde{v}, \cdot, \dots, \cdot).\]
Therefore, if $\alpha$ is multiplicative, then any vector tangent to a fiber of the quotient map is in $\ker \alpha$.

If $\alpha$ is multiplicative, then $d\alpha$ is also multiplicative, and any vector tangent to a fiber of the quotient map is also in $\ker d\alpha$. Since both $\alpha$ and $d\alpha$ annihilate vectors tangent to the fibers, we conclude that $\alpha$ is basic.
\end{proof}

\subsection{Lifting differential forms}
\label{subsec:mapping}
Recall that, for each $n$, the space $\fC(M)$ consists of $C^{p,q}$ bundle maps from $T\Delta^n$ to $T^*M$. For $\varphi \in \fC_n(M)$, a tangent vector at $\varphi$ is given by a $C^{p,q}$ lift $X: T\Delta^n \to TT^*M$ that is linear over $TM$:
\begin{equation}\label{eqn:vectordiagram}
\xymatrix@C=1em{
& & & TT^*M = T^*TM \ar[dd] \ar[dr] & \\
T\Delta^n \ar[dd] \ar^(.3)\varphi[rrrr] \ar^X[urrr] & & & & T^*M \ar[dd] \\
& & & TM \ar[dr] & \\
\Delta^n \ar^{X_0}[urrr] \ar^f[rrrr] & & & & M}
\end{equation}

For each pair of tangent vectors $X,Y \in T_\varphi \fC_n(M)$, we can use the canonical symplectic form $\canomega$ on $T^*M$ to obtain a function $\eta^n_{X,Y}$ on $T\Delta^n$, given by
\begin{equation*}
\eta^n_{X,Y}(v) = \canomega(X(v),Y(v))
\end{equation*}
for $v \in T\Delta^n$.

\begin{prop}\label{prop:etalin}
The function $\eta^n_{X,Y}$ is linear and can therefore be identified with a $1$-form on $\Delta^n$.
\end{prop}
\begin{proof}
The result is a direct consequence of the linearity property of $\canomega$ with respect to the bundle structure of $T^*M \to M$.
\end{proof}

For $n=1$, the operation $(X,Y) \mapsto \int_{\Delta^1} \eta^1_{X,Y}$ is bilinear and skew-symmetric, and so it determines a $2$-form $\omega_1 \in \Omega^2(\fC_1(M))$. For $n=2$, we can also define a $2$-form $\omega_2 \in \Omega^2(\fC_2(M))$ by the formula
\begin{equation}\label{eqn:omega2}
\omega_2(X,Y) = \int_{\Delta^2} d \eta^2_{X,Y} = \int_{\boundary \Delta^2} \eta^2_{X,Y}.
\end{equation}

\begin{prop} \label{prop:coboundary}
$\omega_2$ is the simplicial coboundary of $\omega_1$.
\end{prop}
\begin{proof}
For $i=0,1,2$, let $\sigma_i : \Delta^1 \to \Delta^2$ be the $i$th coface map (which is essentially dual to the face map $d_i$). For any $X,Y \in T_\varphi \fC_2(M)$ and $v \in T\Delta^1$, we have that
\begin{equation}\label{eqn:eta1eta2}
\begin{split}
\eta^1_{Td_i(X),Td_i(Y)}(v) &= \canomega(Td_i(X)(v), Td_i(Y)(v)) \\
&= \canomega(X(T\sigma_i(v)),Y(T\sigma_i(v))) \\
&= \eta^2_{X,Y}(T\sigma_i(v)).
\end{split}
\end{equation}
Using \eqref{eqn:eta1eta2}, we see that
\begin{equation}\label{eqn:faceomega}
\begin{split}
(d_i^* \omega_1)(X,Y) &= \omega_1(Td_i(X),Td_i(Y)) \\
&= \int_{\Delta^1} \eta^1_{Td_i(X),Td_i(Y)} \\
&= \int_{\Delta^1} \sigma_i^* \eta^2_{X,Y}.
\end{split}
\end{equation}
The result then follows from \eqref{eqn:omega2} and \eqref{eqn:faceomega}.
\end{proof}

\begin{prop}\label{prop:exact}
The $2$-forms $\omega_1 \in \Omega^2(\fC_1(M))$ and $\omega_2 \in \Omega^2(\fC_2(M))$ are exact.
\end{prop}
\begin{proof}
Let $\canlam$ denote the tautological $1$-form on $T^*M$, satisfying the property $\canomega = -d\canlam$. We can use $\canlam$ to induce forms on the mapping spaces in a manner similar to the construction of $\omega_1$ and $\omega_2$. Specifically, for $X \in T_\varphi \fC_n(M)$, let $\theta^n_X$ be the function on $T\Delta^n$ given by
\begin{equation*}
\theta^n_X(v) = \canlam(X(v)).
\end{equation*}
Because of the linearity property of $\canlam$, we have that $\theta^n_X$ is a linear function and can therefore be identified with a $1$-form on $\Delta^n$.

Then, let $\lambda_1 \in \Omega^1(\fC_1(M))$  and $\lambda_2 \in \Omega^1(\fC_2(M))$ be defined by
\begin{align*}
\lambda_1(X) &= \int_{\Delta^1} \theta^1_X, & \lambda_2(X) &= \int_{\Delta^2} d\theta^2_X = \int_{\boundary \Delta^2} \theta^2_X.
\end{align*}
The proof of Proposition \ref{prop:coboundary}, with appropriate modification, can be used to show that $\lambda_2 = \delta \lambda_1$.

We claim that $\omega_1 = -d\lambda_1$ (and, since $d$ commutes with $\delta$, therefore $\omega_2 = -d\lambda_2)$. We can check it locally in $M$, as follows.

Let $(x^i, p_i)$ be canonical coordinates on a neighborhood in $T^*M$. Any $C^{p,q}$ bundle map $\varphi: T\Delta^1 \to T^*M$ is locally described by the pullbacks $f^i := \varphi^*(x^i)$ and $\xi_i := \varphi^*(p_i)$, where $f^i \in C^p(\Delta^1)$ and $\xi_i \in C^q_{\mathrm{linear}}(T\Delta^1)$ can be identified with $1$-forms on $\Delta^1$. A tangent vector $X \in T_\varphi \fC_1(M)$ is locally given by $(v^i, \chi_i)$, where the $v^i$ and $\chi_i$ are functions and $1$-forms, respectively, on $\Delta^1$.

We can locally describe $\lambda_1$ by the formula
\begin{equation}\label{eqn:lambdalocal}
\lambda_1|_{(f^i,\xi_i)}(v^i, \chi_i) = \int_{\Delta^1} v^i \xi_i.
\end{equation}
The directional derivatives of $\lambda_1$ are given by
\begin{equation*}
D_{(v^i, \chi_i)} \lambda_1|_{(f^i,\xi_i)}(v'^i , \chi_i') = \int_{\Delta^1} v'^i  \chi_i,
\end{equation*}
from which we obtain the result
\begin{equation*} 
d\lambda_1((v^i, \chi_i),(v'^i , \chi_i')) = \int_{\Delta^1} v'^i  \chi_i - v^i \chi_i' = -\omega_1((v^i, \chi_i),(v'^i , \chi_i')).
\end{equation*}
\end{proof}

\subsection{A symplectic $2$-groupoid}

The notion of \emph{symplectic $2$-groupoid} was defined in \cite{ls:integration, mt:double}. We somewhat imprecisely state the definition as follows:

\begin{definition}\label{def:symp2}
A \emph{symplectic $2$-groupoid} is a Lie $2$-groupoid $\{X_\bullet\}$ that is equipped with a closed, multiplicative $2$-form $\omega \in \Omega^2(X_2)$ satisfying a nondegeneracy condition.
\end{definition}
In Definition \ref{def:symp2}, we have intentionally left the content of the nondegeneracy condition ambiguous. Li-Bland and \v{S}evera \cite{ls:integration} required the $2$-form to be genuinely nondegenerate, so that $(X_2, \omega)$ is a symplectic manifold. In \cite{mt:double}, a weaker condition was stated, where the kernel of $\omega$ is required to be controlled in a certain way by the simplicial structure. For our present purposes, it will suffice to use the genuine nondegeneracy condition; however, since we are dealing with Banach manifolds, we only require weak nondegeneracy.

Since $\omega_2 \in \Omega^2(\fC_2(M))$ is multiplicative, it descends (by Proposition \ref{prop:basic}) to $\LWX_2(M) = \fC_2(M)\modsim$.

\begin{thm}\label{thm:symp2}
$\LWX(M)$, equipped with the $2$-form $\omega_2$, is a symplectic $2$-groupoid.
\end{thm}
\begin{proof}
The fact that $\omega_2$ descends to a closed, multiplicative $2$-form on the truncation is an immediate consequence of Propositions \ref{prop:basic}, \ref{prop:coboundary}, and \ref{prop:exact}. 

It remains to check the nondegeneracy condition. For this, we first observe that $\omega_1 \in \Omega^2(\fC_1(M))$ is (weakly) nondegenerate. Then, using the description of $\LWX_2(M)$ obtained in Lemma \ref{lem:covering}, one can see that a tangent vector in $\LWX_2(M)$ is given by a compatible triplet of tangent vectors in $\fC_1(M)$. If such a triplet $(b_0, b_1, b_2)$ does not vanish everywhere, it is a straightforward exercise to construct another triplet $(b_0', b_1', b_2')$ for which the pairing
\[ \omega_2 ((b_0, b_1, b_2),(b_0', b_1', b_2')) = \omega_1(b_0,b_0') - \omega_1(b_1,b_1') + \omega_1(b_2,b_2') \] 
does not vanish, thereby proving the (weak) nondegeneracy of $\omega_2$ on $\LWX_2(M)$.
\end{proof}

\subsection{3-form twisting}\label{subsec:twisting}

Let $H$ be a $3$-form on $M$. In this section, we describe how $H$ can be used to twist the $2$-forms $\omega_1$ and $\omega_2$. 

For $\varphi \in \fC_n(M)$, let $f: \Delta^n \to M$ be the base map underlying $\varphi: T\Delta^n \to T^*M$. For $X,Y \in T_\varphi \fC_n(M)$, let $X_0, Y_0: \Delta^n \to TM$ be the respective base maps (see \eqref{eqn:vectordiagram}). We can use $H$ to obtain a $1$-form $H^n_{X,Y}$ on $\Delta^n$, given by
\begin{equation*}
H^n_{X,Y}(s) = f^*H(X_0(s), Y_0(s), \cdot)
\end{equation*}
for $s \in \Delta^n$. We then define $2$-forms $\phi^H_1 \in \Omega^2(\fC_1(M))$ and $\phi^H_2 \in \Omega^2(\fC_2(M))$ by
\begin{align*}
\phi^H_1(X,Y) &= \int_{\Delta^1} H^1_{X,Y}, \\
\phi^H_2(X,Y) &= \int_{\Delta^2} dH^1_{X,Y}.
\end{align*}

\begin{remark}
The forms $\phi^H_i$, $i=1,2$, only depend on the information about the underlying base maps, so are actually pullbacks of forms on $S_i(M) := C^p(\Delta^i,M)$. The construction of $\phi^H_1$ is a special case of a more general transgression procedure taking any $\beta \in \Omega^p(M)$ to $\phi^\beta_i \in \Omega^{p-1}(S_i(M))$.
\end{remark}

\begin{prop}\label{prop:phicoboundary}
$\phi^H_2$ is the simplicial coboundary of $\phi^H_1$.
\end{prop}
\begin{proof}
The result follows from an argument similar to the proof of Proposition \ref{prop:coboundary}.
\end{proof}

\begin{prop}\label{prop:dphi}
If $H$ is closed, then $d\phi^H_1$ is the simplicial coboundary of $H$.
\end{prop}
\begin{proof}
In local coordinates on $M$, write $H = \frac{1}{6} H_{ijk} dx^i \wedge dx^j \wedge dx^k$. Then, in the local neighborhood, for $f = (f^i) \in C^p(\Delta^1;M)$ and any $C^p$ sections $X_0 = (X_0^i), Y_0 = (Y_0^i)$ of $f^*(TM)$, we have
\[ \phi^H_1|_f (X_0,Y_0) = \int_{\Delta^1} f^*(H_{ijk}) X_0^i Y_0^j df^k.\]
The differential of $\phi^H_1$ is then given by
\begin{equation}\label{eqn:dphi}
d\phi^H_1|_f(X_0,Y_0,Z_0) = \int_{\Delta^1} X_0^*(dH_{ijk}) Y_0^i Z_0^j df^k + f^*(H_{ijk})Y_0^i Z_0^j dX_0^k + \cycl.
\end{equation}
The integral of the first term on the right side of \eqref{eqn:dphi}, together with its cyclic permutations, is equal to
\begin{equation}\label{eqn:dphi1}
 \phi^{dH}_1|_f(X_0,Y_0,Z_0) + \int_{\Delta^1} f^*(dH_{ijk})X_0^i Y_0^j Z_0^k. 
\end{equation}
The integral of the second term on the right side of \eqref{eqn:dphi}, together with its cyclic permutations, is
\begin{equation}\label{eqn:dphi2}
 \int_{\Delta^1} f^*(H_{ijk})d(X_0^i Y_0^j Z_0^k).
\end{equation}
Putting \eqref{eqn:dphi1} and \eqref{eqn:dphi2} together, we have
\begin{equation*}
 \begin{split}
 d\phi^H_1|_f(X_0,Y_0,Z_0) &=  \phi^{dH}_1|_f(X_0,Y_0,Z_0) +  \int_{\Delta^1} d\left(f^*(H_{ijk})X_0^i Y_0^j Z_0^k\right) \\
 &= \phi^{dH}_1|_f(X_0,Y_0,Z_0) + \left[ f^*(H_{ijk})X_0^i Y_0^j Z_0^k \right]^1_0,
 \end{split}
\end{equation*}
or, in other words,
\begin{equation}
 d\phi^H_1 = \phi^{dH}_1 + \delta H.
\end{equation}
In particular, if $H$ is closed, then $d\phi^H_1 = \delta H$.
\end{proof}

Since $\phi_2^H \in \Omega^2(\fC_2(M))$ is multiplicative, it descends to $\LWX_2(M)$. We now arrive at the main result of this section.

\begin{thm}\label{thm:symp2twist}
Let $M$ be a manifold, and let $H$ be a closed $3$-form on $M$. Then $\LWX_2(M)$, equipped with the $2$-form $\omega_2^H := \omega_2 + \phi^H_2$, is a symplectic $2$-groupoid.
\end{thm}
\begin{proof}
An immediate consequence of Propositions \ref{prop:phicoboundary} and \ref{prop:dphi} is that $\phi^H_2$ is closed and multiplicative. We therefore have that $\omega_2^H$ is closed and multiplicative.

As in the proof of Theorem \ref{thm:symp2}, the nondegeneracy property follows from the observation that $\omega_1^H := \omega_1 + \phi^H_1$ is nondegenerate. This can be seen by showing that, for any $X \in T_\varphi \fC_1(M)$, one can construct $Y \in T_\varphi \fC_1(M)$ for which $\phi^H_1(X,Y) = 0$ and $\omega_1(X,Y) \neq 0$. We leave the details as an exercise. 
\end{proof}

\section{Integration of Dirac structures}\label{sec:dirac}
In this section, we study the geometry of integration of Dirac structures in relation to the Liu-Weinstein-Xu $2$-groupoid.

\subsection{$A$-path integration} \label{sec:a-path}
Let ${\mathcal {D}}$ be a Dirac structure in an exact Courant algebroid $(TM\oplus T^*M, H)$. Letting $\rho$ be the canonical projection from $TM\oplus T^*M$ to $TM$, we have that $(\cald, \rho)$ forms a Lie algebroid. The integration of $(\cald, \rho)$ via ``$A$-paths'' was studied in \cite{cf:integration, bcwz:dirac}. 

In this section, we will construct a simplicial manifold $\{\frakg_\bullet(\mathcal{D})\}$ that connects the $A$-path integration of a Dirac structure to $\{\mathfrak{C}_\bullet(M)\}$. First, we briefly review the $A$-path construction. 

Define $P(\mathcal{D})$ to be the space of $C^{2,1}$ paths $\alpha: \Delta^1\to T^*M$  satisfying
\[
\left(\frac{d}{dt}(\pi\circ \alpha)(t), \alpha(t)\right)\in \cald,\ \forall t\in \Delta^1. 
\]
It is proved in \cite[Lemma 4.6]{cf:integration} that $P(\mathcal{D})$, which is called the space of \emph{$A$-paths}, is a Banach manifold. 

Recall that $\fC_1(M)$ can be identified with the space of all $C^{p,q}$ maps from $\Delta^1 \to T^*M$. Taking $p=2$ and $q=1$, we have a natural, smooth inclusion map $\iota_1:P(\mathcal{D})\to \mathfrak{C}_1(M)$. 

To construct a groupoid integrating $\cald$, one needs to impose the homotopy relation on $A$-paths; we refer to \cite[Definition 1.4]{cf:integration} for the precise definition. Crainic and Fernandes proved \cite[Theorem 2.1]{cf:integration} that the quotient $\sG:=P(\mathcal{D})\modsim$ is a source-simply-connected topological groupoid. In general, $\sG$ could fail to be smooth, and necessary and sufficient conditions for $\cald$ to be integrable to a Lie groupoid were obtained in \cite[Theorem 4.1]{cf:integration}.

We will now describe a simplicial manifold associated to a Dirac structure $\mathcal{D}$. For simplicity, we will assume that $\cald$ is integrable to a Lie groupoid, although many of the results will carry through in the general case. 

For each $n \geq 0$, let $\fG_n(\cald)$ denote the set of $C^2$ groupoid morphisms from the pair groupoid $\Delta^n \times \Delta^n$ to $\sG$. There is a natural simplicial structure on $\{\fG_\bullet(\cald)\}$, induced by the cosimplicial structure of $\{\Delta^\bullet\}$.

\begin{lemma}\label{lem:frakd}The space $\frakg_n(\mathcal{D})$ is a Banach manifold. 
\end{lemma}

\begin{proof}
 Given a $C^2$ groupoid morphism $\Sigma: \Delta^n \times \Delta^n \to \sG$, we define a $C^2$ map $\sigma: \Delta^n \to \sG$ by the equation $\sigma(w) = \Sigma(w,0)$. We observe that $\sigma$ satisfies the following two properties:
 \begin{itemize}
  \item $\sigma(0)$ is a unit of $\sG$,
  \item $\sigma(w)$ is in the same source-fiber as $\sigma(0)$ for all $w \in \Delta^n$.
 \end{itemize}
Conversely, given any $C^2$ map $\sigma: \Delta^n \to \sG$ satisfying the above properties, we may obtain $\Sigma \in \fG_n(\cald)$ by setting $\Sigma(w_1,w_2) = \sigma(w_1)\sigma(w_2)^{-1}$, so we have a one-to-one correspondence.

We will now show that the space of $\sigma$ satisfying the above properties (and hence $\fG_n(\cald)$) is a Banach manifold. Let $s: \sG \to M$ denote the source map. Since $s$ is a submersion, we have that $s^{-1}(x)$ is a submanifold of $\sG$ for each $x \in M$, so $C^2(\Delta^n; s^{-1}(x))$ is a Banach manifold.

Consider the evaluation map $ev_0: C^2(\Delta^n; s^{-1}(x))\to s^{-1}(x)$, defined by $ev_0(f):=f(0)$ for $f\in C^2(\Delta^n; s^{-1}(x))$. It is not difficult to check that $ev_0$ is a surjective submersion between Banach manifolds. Therefore, $C^2(\Delta^n; s^{-1}(x))_0:=ev_0^{-1}(x)$ is a Banach manifold.  

As $\sigma(\Delta^n)$ is compact and the map $s:\mathsf{G}\to M$ is submersive, we have that $\frakg (\cald)_n$ near $\sigma$ is locally a product of a neighborhood of $\sigma$ in $C^2(\Delta^n, s^{-1}(x))_0$ and $\mathbb{R}^{\operatorname{dim}(M)}$. This shows that $\frakg_n(\mathcal{D})$ is a Banach manifold.  
\end{proof}
 
\begin{remark}\label{rmk:hausdorff} We point out that, since $\mathsf{G}$ might not be a Hausdorff manifold \cite{cf:integration}, it is possible for $\frakg_n(\mathcal{D})$ to be non-Hausdorff as well. 
\end{remark}

Following the approach of the proof of Lemma \ref{lem:frakd}, one can show that the horn maps are surjective submersions, thereby completing the proof of the following statement.
\begin{prop}\label{prop:dirackan} $\{\frakg_\bullet(\mathcal{D})\}$ is a Kan simplicial manifold. 
\end{prop}

For any $\Sigma\in \frakg_n(\mathcal{D})$, we may apply the Lie functor to obtain a $C^{2,1}$ Lie algebroid morphism $\bar{\Sigma}$ from $T\Delta^n$ to $\mathcal{D}$; in fact, this process gives a one-to-one correspondence, since $\sG$ and the pair groupoid $\Delta^n \times \Delta^n$ are both source-simply-connected. 

When $n=1$, this correspondence allows us to identify $\frakg_1(\mathcal{D})$ with the $A$-path space $P(\mathcal{D})$. Furthermore, when we consider the truncation $\tau_{\leq 1}\frakg(\mathcal{D})$, the equivalence imposed by the truncation corresponds to homotopy equivalence of $A$-paths (see \cite[Propositions 1.1, 1.3]{cf:integration}). In other words, $\left(\tau_{\leq 1}\frakg(\mathcal{D})\right)_1$ can be identified with $\mathsf{G} = P(\mathcal{D}) \modsim$. Thus we see that $\sG$ can be recovered from $\{\fG_\bullet(\cald)\}$ by truncation.

To connect $\{\fG_\bullet(\cald)\}$ to $\{\fC_\bullet(M)\}$, let $\pi_{T^*}$ be the canonical projection from $TM\oplus T^*M$ to $T^*M$. The map taking $\Sigma \in \fG_n(\cald)$ to the bundle map $\pi_{T^*}\circ \bar{\Sigma} : T\Delta^n \to T^*M$ defines a natural map of simplicial manifolds
\begin{equation}\label{eq:dirac-courant}
F_\bullet: \frakg_\bullet(\mathcal{D}) \to \mathfrak{C}_\bullet(M).  
\end{equation}
Superficially, it may seem that $F_\bullet$ discards the information about the $TM$-component of $\bar{\Sigma}$; however, the fact that $\bar{\Sigma}$ is a Lie algebroid morphism implies that the $TM$-component can be recovered by applying the tangent functor to the underlying map $\Delta^n \to M$. We therefore see that $F_\bullet$ is injective.

The following diagram of Kan simplicial manifolds summarizes the various relationships we have described.
\begin{equation}\label{diag:simplicial-dirac}
\xymatrix{ \fG_\bullet(\cald) \ar@{^{(}->}^{F_\bullet}[r] \ar_{\tau_{\leq 1}}[d] & \fC_\bullet(M) \ar^{\tau_{\leq 2}}[d] \\
\sG & \LWX(M)}
\end{equation}

\subsection{Presymplectic $2$-form}

\label{subsec:presymp}

In Section \ref{subsec:mapping}, we introduced $2$-forms $\omega_i$, as well as their twisted versions $\omega^H_i$, on $\mathfrak{C}_i(M)$ for $i=1,2$.   In this subsection, we study the relationship of these $2$-forms to the simplicial manifold $\{\frakg_\bullet(\cald)\}$ associated to a Dirac structure $\cald$. Our main results are as follows.

\begin{thm}
\label{thm:omega2pullback} The $2$-form $F_2^*\omega^H_2 \in \Omega^2(\frakg_2(\cald))$ vanishes.
\end{thm}
Together with the results of Section \ref{sec:symplectic} (specifically, Propositions \ref{prop:basic}, \ref{prop:phicoboundary}, and \ref{prop:dphi}), Theorem \ref{thm:symp2twist} implies the following result.
\begin{cor}
\label{cor:omega1pullback}
The $2$-form $F_1^* \omega^H_1$ is closed and multiplicative, and it therefore descends to a closed, multiplicative $2$-form on the Lie groupoid $\sG$ integrating $\cald$.
\end{cor}

\begin{remark}
In Corollary \ref{cor:omega1pullback}, we recover one of the main results of \cite{bcwz:dirac}. However, \cite{bcwz:dirac} showed that the $2$-form on $\sG$ satisfies an additional property that controls the extent to which it is degenerate. It remains unclear how this property arises from the inclusion \eqref{eq:dirac-courant}, but we expect that it is related to the Lagrangian property discussed in Section \ref{sec:lagrangian}.
\end{remark}

In the remainder of this section, we will prove Theorem \ref{thm:omega2pullback}. For simplicity, we assume $H=0$. The extension to the general case is straightforward and left to the reader. 

It suffices to check the statement locally on a coordinate chart of $M$. On such a chart, let $(x^i, q^i, p_i)$ be coordinates on $TM\oplus T^*M$.  A $C^{2,1}$ bundle map $\varphi:T\Delta^2\to T^*M$ is locally given by $\varphi = (f^i, \xi_i)$, where $f^i:=\varphi^*(x^i)$ is a function on $\Delta^2$ and $\xi_i:=\varphi^*(p_i)$ is a $1$-form on $\Delta^2$ for each $i$. Together, the $f^i$ form a $C^2$ map $f:\Delta^2\to M$, and the $\xi_i$ form a $C^1$ element of $\Omega^1\left(\Delta^2, f^*(T^*M)\right)$.

A tangent vector on $\mathfrak{C}_2(M)$ at $\varphi$ is locally given by $(v^i, \chi_i)$, where the $v^i$ describe a section of $f^*(TM)$ and the $\chi_i$ describe an element of $\Omega^1\left(\Delta^2, f^*(T^*M)\right)$. In the above coordinates, $\eta_2$ has the form
\[
\eta_2\left((v^i, \chi_i),(v'^i, \chi'_i)\right)=v^i\chi_i'-v'^i\chi_i,
\]
so the 2-form $\omega_2$ on $\mathfrak{C}_2(M)$ is given by
\begin{equation}\label{eqn:omegalocal}
\omega_2((v^i, \chi_i),(v'^i, \chi'_i))=\int_{\Delta^2} d(v^i\chi_i'-v'^i\chi_i).
\end{equation}

Let $n = \dim M$. Since the rank of the Dirac structure $\cald$ is $n$, we can locally find linearly independent sections $\Theta_\alpha = Q_\alpha + P_\alpha = q^i_\alpha \partial_i + p_{i\alpha} dx^i$, for $\alpha = 1, \dots, n$, that span $\mathcal{D}$. The following properties hold by definition of Dirac structures:
\begin{align}
\langle \Theta_\alpha, \Theta_\beta\rangle &= q^i_\alpha p_{i\beta}+q^i_\beta p_{i\alpha}=0, \label{eqn:isotropy} \\
\left[\Theta_\alpha, \Theta_\beta\right] &= C_{\alpha\beta}^\gamma \Theta_\gamma, \label{eqn:integrability}
\end{align}
where $C_{\alpha\beta}^\gamma$ is a smooth function on $M$, and $\left[\cdot,\cdot\right]$ is the Courant bracket.  Using the definition of the Courant bracket, \eqref{eqn:integrability} implies that
\begin{equation}\label{eqn:int2}
 C_{\alpha\beta}^\gamma P_\gamma = L_{Q_\alpha} P_\beta - \iota_{Q_\beta} dP_\alpha,
 \end{equation}
which in coordinates becomes
\begin{equation}\label{eq:courant}
 C_{\alpha\beta}^\gamma p_{i\gamma} = q^j_\alpha \partial_j(p_{i\beta}) - q^j_\beta \partial_j(p_{i\alpha}) + p_{j\beta}\partial_i(q^j_\alpha) + q^j_\beta \partial_i(p_{j \alpha}).
\end{equation}

Recall from Section \ref{sec:a-path} that a point $\Psi$ of $\mathfrak{G}_2(\mathcal{D})$ can be identified with a $C^{2,1}$ Lie algebroid morphism from $T\Delta^2$ to $\mathcal{D}$. A $C^{2,1}$ bundle map from $T\Delta^2$ to $\cald$ can be locally described by a $C^2$ map $f:\Delta^2\to M$ and $C^1$ elements $\psi^\alpha\in\Omega^1(\Delta^2) $ where, for $v \in T\Delta^2$,
\[
\Psi(v) = \psi^\alpha(v) \Theta_\alpha.
\]
Using the characterization of Lie algebroid morphisms in terms of differentials, we have that a bundle map $\Psi:T\Delta^2\to \mathcal{D}$ is a Lie algebroid morphism if and only if
\begin{align}
df^i &= f^*(q^i_\alpha)\psi^\alpha, \label{eqn:liealg1} \\
d\psi^\gamma &= -\frac{1}{2} f^*(C^\gamma_{\alpha\beta})\psi^\alpha \wedge \psi^\beta. \label{eqn:liealg2}
\end{align}
Given $\Psi \in \frakg_2(\cald)$, the induced bundle map $\hat{\Psi} := F_2 (\Psi)$ from $T\Delta^2$ to $T^*M$ is given by 
\[ \hat{\Psi} = (f^i, f^*(p_{i \alpha}) \psi^\alpha).\]

A tangent vector on $\frakg_2(\cald)$ at $\Psi$ is given by a collection of $C^2$ functions $v^i$ on $\Delta^2$, representing a vector field along $f$, and $C^1$ $1$-forms $\mu^\alpha$, satisfying the following equations, which are obtained by differentiating \eqref{eqn:liealg1} and \eqref{eqn:liealg2}:
\begin{align}
dv^i &= v^j f^*(\partial_j q^i_\alpha) \psi^\alpha + f^*(q^i_\alpha) \mu^\alpha, \label{eq:dchi} \\
d\mu^\gamma &= -\frac{1}{2} v^j f^*(\partial_j C^\gamma_{\alpha \beta}) \psi^\alpha \wedge \psi^\beta - f^*(C^\gamma_{\alpha \beta}) \mu^\alpha \wedge \psi^\beta. \label{dpsi}
\end{align}
The induced tangent vector on $\fC_2(M)$ is given by $(v^i,\chi_i)$, where
\begin{equation}\label{eqn:inducedx}
\chi_i = v^k f^*(\partial_k p_{i\alpha}) \psi^\alpha + f^*(p_{i\alpha})\mu^\alpha.
\end{equation}
Putting this into \eqref{eqn:omegalocal}, we obtain the formula
\[(F_2^*\omega_2)\left((v^i,\mu^\alpha),(v'^i,\mu'^\alpha)\right) = \int_{\Delta^2} d\Xi,\]
where
\begin{equation}\label{eqn:Xi}
\Xi = f^*(\partial_k p_{i\alpha})\psi^\alpha (v^i v'^k - v'^i v^k) + f^*(p_{i\alpha}) (v^i \mu'^\alpha - v'^i\mu^\alpha)
\end{equation}
is a $1$-form on $\Delta^2$. We claim that $d\Xi = 0$, which will imply Theorem \ref{thm:omega2pullback}.

The proof will proceed as follows. First, we can use \eqref{eqn:liealg1}--\eqref{dpsi} to write $d\Xi$ in terms of $f$, $\psi^\alpha$, $v^i$, $v'^i$, $\mu^\alpha$, $\mu'^\alpha$, and the various structure functions. Second, we can collect terms that are of similar type with respect to the $\mu$'s and $\psi$'s. We will see that each group of terms vanishes as a result of \eqref{eqn:isotropy}--\eqref{eq:courant}.

\subsubsection{Terms of type $\mu \wedge \mu$}
In $d\Xi$, the coefficient of $\mu^\alpha \wedge \mu'^\beta$ is $f^*(p_{i\alpha} q^i_\beta + p_{i\beta} q^i_\alpha) = f^*(\langle \Theta_\alpha, \Theta_\beta \rangle)$, which vanishes by the isotropy condition \eqref{eqn:isotropy}.

\subsubsection{Terms of type $\psi \wedge \mu$}
The coefficient of $\psi^\alpha \wedge \mu^\beta$ in $d\Xi$ is
\begin{equation*} 
\begin{split}
& f^*\left(\partial_j(p_{i\alpha})q^j_\beta - \partial_i(p_{j\alpha}) q^j_\beta - q^j_\alpha \partial_j(p_{i\beta}) - p_{j\beta} \partial_i(q^j_\alpha) + p_{i\gamma}C^\gamma_{\alpha \beta}\right) v'^i \\
&= \left(C_{\alpha \beta}^\gamma P_\gamma - L_{Q_\alpha} P_\beta + \iota_{Q_\beta} dP_\alpha \right) (v'),
\end{split}
\end{equation*}
which vanishes by the integrability condition \eqref{eqn:int2}--\eqref{eq:courant}. Because of skew-symmetry, the coefficient of $\psi^\alpha \wedge \mu'^\beta$ will similarly vanish.

\subsubsection{Terms of type $\psi \wedge \psi$}
The coefficient of $\psi^\alpha \wedge \psi^\beta$ is
\begin{multline*}
f^*\left(q_\alpha^k \partial_k \partial_j(p_{i\beta}) - q_\beta^k \partial_k \partial_j(p_{i\alpha}) - \partial_j(p_{i\gamma})C_{\alpha \beta}^\gamma + \partial_j(p_{k\beta})\partial_i(q_\alpha^k) \right. \\
\left. - \partial_j(p_{k\alpha})\partial_i(q_\beta^k) + \partial_k(p_{i\beta})\partial_j(q_\alpha^k) - \partial_k(p_{i\alpha})\partial_j(q_\beta^k) 
- p_{i\gamma}\partial_j(C_{\alpha \beta}^\gamma)\right)v^iv'^j,
\end{multline*}
plus terms that are antisymmetric in $i,j$. This is equal to
\[ f^* \partial_j \left(q_\alpha^k \partial_k(p_{i\beta}) - q_\beta^k \partial_k(p_{i\alpha}) + p_{k\beta} \partial_i(q_\alpha^k) + \partial_i(p_{k\alpha})q_\beta^k - p_{i\gamma}C_{\alpha \beta}^\gamma\right)v^i v'^j,\]
again plus terms that are antisymmetric in $i,j$. We can recognize this expression as
\[ d\left(L_{Q_\alpha} P_\beta - \iota_{Q_\beta} dP_\alpha - C_{\alpha \beta}^\gamma P_\gamma\right)(v^i, v'^j),\]
which vanishes by the integrability condition \eqref{eqn:int2}--\eqref{eq:courant}.

\begin{remark}
We observe that the proof of Theorem \ref{thm:omega2pullback} does not rely on the assumption that $\cald$ is \emph{maximally} isotropic. Therefore, the result applies to any isotropic subbundle of $TM \oplus T^*M$ satisfying the integrability condition \eqref{eqn:integrability}. The next section aims to address the question of what distinguishes the isotropic case from the maximally isotropic case.
\end{remark}

\section{Dirac structures and Lagrangian sub-$2$-groupoids}\label{sec:lagrangian}

In Section \ref{sec:dirac}, we considered the geometry of the inclusion $F_\bullet: \frakg_\bullet(\cald)\hookrightarrow \fC_\bullet(M)$, when $\cald$ is a Dirac structure. Consider the composition of this map with the truncation map $\tau_{\leq 2}: \fC(M) \to \LWX(M)$, and let $L_\cald$ be the image of $\frakg_2(\cald)$ in $\LWX_2(M)$. The image $L_\cald$ determines a sub-$2$-groupoid $\mathfrak{L}_\cald$ of $\LWX(M)$, and the $1$-truncation of $\mathfrak{L}_\cald$ can be identified with the $1$-truncation of $\frakg(\cald)$, which, as we noted in Section \ref{sec:dirac}, recovers the Lie groupoid $\sG$ integrating $\cald$. We would now like to consider the geometry of the sub-$2$-groupoid $\mathfrak{L}_\cald \subset \LWX(M)$.

For $x\in M$, the zero bundle map $\Psi^x:T\Delta\to \cald$ over the constant map $f^x:\Delta\to M$, $s \mapsto x$, defines a point in $L_\cald$.  This defines an embedding $M\hookrightarrow L_\cald \subset \LWX_2(M)$. We can think of this image of $M$ as being the space of ``units'', since it is equal to the image of $M$ under the ``double degeneracy'' maps.

\begin{prop}\label{prop:lagrangian} For all $x\in M$, $T_x L_\cald$ is a Lagrangian subspace of $T_x \LWX_2(M)$.
\end{prop}

\begin{proof}
By Theorem \ref{thm:omega2pullback}, we already know that $L_\cald$ is isotropic. It remains to show that $T_x L_\cald$ is coisotropic.

We use the local description and notation from Section \ref{subsec:presymp}. Choose coordinates on $M$ for which $x=0$. Then, if we write $\Psi^x = (f^i, \psi^\alpha)$ as in Section \ref{subsec:presymp}, we have $f^i = 0$ and $\psi^\alpha = 0$.

Let $(v^i, \mu^\alpha)$ be a tangent vector on $\mathfrak{G}_2(\cald)$ at $\Psi^x$. In this case, \eqref{eq:dchi} and \eqref{dpsi} reduce to
\begin{align}
dv^i &= q^i_\alpha(0) \mu^\alpha, \label{eq:chix}\\
d\mu^\alpha &= 0.\label{eq:psix}
\end{align}
Any solution to \eqref{eq:chix}--\eqref{eq:psix} can be written in the form
\begin{align}
v^i &= q^i_\alpha(0) g^\alpha + c^i, \label{eqn:vi}\\
\mu^\alpha &= dg^\alpha,
\end{align}
for some $C^2$ functions $g^\alpha$ on $\Delta^2$ and constants $c^i$. From \eqref{eqn:inducedx}, we have that the induced tangent vector on $\fC_2(M)$ has
\begin{equation}
\chi_i = p_{i\alpha}(0)\mu^\alpha = p_{i\alpha}(0)dg^\alpha.\label{eqn:chii}
\end{equation}

At the level of tangent vectors, the truncation map $\tau_{\leq 2}$ has the effect of pulling back to $\boundary\Delta^2$. In what follows, let $j: \boundary \Delta^2 \to \Delta^2$ be the natural inclusion map. 

To prove that $T_x L_\cald$ is coisotropic, we will show that, if any tangent vector $(v'^i,\chi'_i) \in T_x \fC_2(M)$ is such that $\omega_2^H((v^i, \chi_i), (v'^i, \chi'_i))=0$ for all $(v^i,\chi_i)$ of the form \eqref{eqn:vi},\eqref{eqn:chii}, then $(j^*v'^i,j^*\chi'_i)$ takes the same form.

We compute
\begin{equation*}
\begin{split}
\omega^H_2\left((v^i, \chi_i),(v'^i,\chi'_i)\right) 
&= \int_{\Delta^2} d\left((q^i_\alpha(0)g^\alpha + c^i)\chi'_i - v'^i p_{i\alpha}(0) dg^\alpha\right)\\
&= \int_{\boundary \Delta^2} c^i \chi'_i + \int_{\boundary \Delta^2} \left(q^i_\alpha(0)\chi'_i + p_{i\alpha}(0)dv'^i\right)g^\alpha.
\end{split}
\end{equation*}
The requirement that this vanishes for all $g^\alpha$ and $c^i$ implies that
$\int_{\boundary \Delta^2} \chi'_i$ vanishes for all $i$, and that $j^*(q^i_\alpha(0) \chi'_i + p_{i\alpha}(0)dv'^i)$ vanishes for all $\alpha$. From the first condition, we have that $j^*\chi'_i$ is exact, so let $\Lambda_i$ be functions on $\boundary \Delta^2$ such that $j^*\chi'_i = d\Lambda_i$. From the latter condition, we then have that
\begin{equation}
\begin{split}
 0 &= d(q^i_\alpha(0) \Lambda_i + p_{i\alpha}(0)j^* v'^i)\\
 &= d \langle \Theta_\alpha(0), j^* v'^i \partial_i + \Lambda_i dx^i \rangle. \label{eqn:dlambda}
 \end{split}
 \end{equation}
Let $0$ denote the $0$th vertex of $\Delta^2$, and let $e^i = v'^i(0)$, $\varepsilon_i = \Lambda_i(0)$. Then \eqref{eqn:dlambda} implies that 
\[ (j^* v'^i - e^i)\partial_i + (\Lambda_i - \varepsilon_i)dx^i \]
annihilated $\cald$. Using the fact that $\cald$ is maximally isotropic, we deduce that there exist unique functions $\Phi^\alpha$ on $\boundary \Delta^2$ such that
\[ (j^* v'^i - e^i)\partial_i + (\Lambda_i - \varepsilon_i)dx^i = \Phi^\alpha \Theta_\alpha(0),\]
implying that
\begin{align}
j^* v'^i &= q^i_\alpha(0)\Phi^\alpha  + e^i, \label{eqn:jv} \\
\Lambda_i &= p_{i\alpha}(0)\Phi^\alpha  + \varepsilon_i. \label{eqn:lambda}
\end{align}
Differentiating the latter equation, we have
\begin{equation}\label{eqn:jchi}
j^*\chi'_i =  p_{i\alpha}(0)d\Phi^\alpha.
\end{equation}
From \eqref{eqn:jv} and \eqref{eqn:jchi}, we see that $j^*v'^i$ and $j^*\chi'_i$ indeed take the desired form of \eqref{eqn:vi} and \eqref{eqn:chii}.
\end{proof}

Proposition \ref{prop:lagrangian} suggests the following conjecture.
\begin{conj}\label{conj:lagrangian}
$\mathfrak{L}_\cald$ is a Lagrangian sub-$2$-groupoid of the symplectic $2$-groupoid $\LWX(M)$, i.e. $L_D$ is a Lagrangian submanifold of $\LWX_2(M)$.
\end{conj}

It is well-known that, if $B \in \Omega^2(M)$ is a closed $2$-form on $M$, then the graph of $B^\flat : TM \to T^*M$ is a Dirac structure (with $H=0$). We prove Conjecture \ref{conj:lagrangian} in this special case.

\begin{prop}\label{prop:presym}Let $B$ be a closed $2$-form on $M$, and let $\cald \subset TM \oplus T^*M$ be the graph of $B^\flat$. Then $\mathfrak{L}_\cald$ is a Lagrangian sub-$2$-groupoid of $\LWX(M)$.
\end{prop}
\begin{proof}
Because of Theorem \ref{thm:omega2pullback}, we only need to show that $L_\cald$ is coisotropic. In the notation of Section \ref{subsec:presymp}, we can take the local trivialization of $\cald$ to be given by $\Theta_i = \partial_i + B_{ij} dx^j$, where $B = \frac{1}{2}B_{ij}dx^i \wedge dx^j$. In this frame, the structure functions $C_{ij}^k$ vanish.

Using \eqref{eqn:liealg1}, \eqref{eqn:liealg2}, and the above description of $\Theta_i$, we have that a point $\Psi \in \frakg_2(\cald)$ is given by $(f^i, \psi^i)$, where $df^i = \psi^i$ and $\psi^i = 0$ (of course, the latter condition is redundant). Similarly, equations \eqref{eq:dchi} and \eqref{dpsi}, describing a tangent vector $(v^i, \mu^i)$ at $\Psi$, reduce to $dv^i = \mu^i$, $d\mu^i = 0$. From \eqref{eqn:inducedx}, we have that the induced tangent vector on $\fC_2(M)$ is of the form
\begin{equation}\label{eqn:chiib}
    \chi_i = v^k f^*(\partial_k B_{ij})\psi^j + f^*(B_{ij})\mu^j = v^k f^*(\partial_k B_{ij})df^j + f^*(B_{ij})dv^j.
\end{equation}

As in the proof of Proposition \ref{prop:lagrangian}, let $j: \boundary \Delta^2 \to \Delta^2$ be the natural inclusion map. To prove that $L_\cald$ is coisotropic, we will show that, if any tangent vector $(v'^i, \chi'_i)$ is such that $\omega_2((v^i, \chi_i),(v'^i, \chi'_i)) = 0$ for all $v^i$, with $\chi_i$ of the form \eqref{eqn:chiib}, then $j^*\chi'_i$ takes the same form. We compute

\begin{equation*}
\begin{split}
 \omega_2\left((v^i,\chi_i),(v'^i,\chi'_i)\right) &= \int_{\boundary \Delta^2} v^i \chi'_i - v'^i\left(v^k f^*(\partial_k B_{ij}) df^j + f^*(B_{ij})dv^j\right) \\
 &= \int_{\boundary \Delta^2} v^i \left( \chi'_i - v'^k f^*(\partial_i B_{kj})df^j - d(v'^j f^*(B_{ij})) \right).
\end{split}
\end{equation*}
The requirement that this vanish for all $v^i$ implies that
\begin{equation}\label{eqn:jstarchi}
    \begin{split}
 j^* \chi'_i &= j^* \left( v'^k f^*(\partial_i B_{kj})df^j + d(v'^j f^*(B_{ij})) \right)  \\
 &= j^* \left( v'^k f^*(\partial_i B_{kj})df^j + v'^j f^*(\partial_k B_{ij})df^k + f^*(B_{ij})dv'^j  \right).
    \end{split}
\end{equation}
Using the fact that $B$ is closed, we may rewrite \eqref{eqn:jstarchi} as
\[  j^* \chi'_i = j^* \left( v'^k f^*(\partial_k B_{ij})df^j + f^*(B_{ij})dv'^j  \right),\]
showing that $j^* \chi'_i$ takes the desired form.
\end{proof}

\bibliographystyle{amsplain}
\bibliography{courantbib}

\end{document}